\documentclass[12pt]{article}
\usepackage{amsmath,amsthm,fullpage,graphicx}
\usepackage{amssymb}

\usepackage{comment}

\newtheorem{theorem}{Theorem}[section]
\newtheorem{lemma}[theorem]{Lemma}
\newtheorem{corollary}[theorem]{Corollary}

\newtheorem{proposition}[theorem]{Proposition}

\newtheorem{example}[theorem]{Example}

\newtheorem{tabl}[theorem]{Table}

\renewcommand{\geq}{\geqslant}
\renewcommand{\leq}{\leqslant}
\renewcommand{\ge}{\geqslant}
\renewcommand{\le}{\leqslant}
\def\Z{\mathbb{Z}}

\def\cref#1{Corollary~$\ref{#1}$}

\input epsf

\usepackage{graphicx}
\graphicspath{%
    {converted_graphics/}
    {/}
    {C:/Users/Jeff/Desktop/}
    {C:/Users/Jeff/Dropbox/Heffter/converted_graphics/}
    {C:/Users/Jeff/Dropbox/Heffter/}
    {C:/Users/Jeff/Dropbox/Photos/}
}
\begin{document}


\title{Square Integer Heffter Arrays with Empty Cells}

\author{
 D.S. Archdeacon \\
      Dept. of Math. and Stat. \\
      University of Vermont \\
      Burlington, VT 05405 \ \ USA\\
      {\tt dan.archdeacon@uvm.edu} \\ \and
J.H.~Dinitz\\
Dept. of Math. and Stat.\\
University of Vermont\\
 Burlington, VT 05405 \ \ USA\\
{\tt jeff.dinitz@uvm.edu}\\  \\ \and
D.M. Donovan\\
Department of Mathematics\\
The University of Queensland\\
Brisbane, QLD Australia\\
{\tt dmd@maths.uq.edu.au}\\
\and
Emine \c{S}ule Yaz{\i}c{\i}\\
Department of Mathematics,\\ Ko\c{c} University, Sar{\i}yer,\\
34450, \.{I}stanbul, Turkey\\
{\tt eyazici@ku.edu.tr}\\
}

\maketitle

\begin{center}
{\em Dedicated to the memory of our friend and colleague Scott Vanstone.}\\
\vspace*{1.0ex}
\end{center}
\begin{abstract} A Heffter array $H(m,n;s,t)$ is an $m \times n$ matrix with nonzero entries from $\mathbb{Z}_{2ms+1}$ such that {\em i}) each row contains $s$ filled cells and each column contains $t$ filled cells, {\em ii}) every row and column sum to 0, and {\em iii}) no element from $\{x,-x\}$ appears twice. Heffter arrays are useful in embedding  the complete graph $K_{2nm+1}$ on an orientable surface where the embedding has the property that each edge borders exactly one  $s-$cycle and one $t-$cycle.  Archdeacon, Boothby and Dinitz proved that these arrays can be constructed in the case when $s=m$, i.e every cell is filled.  In this paper we concentrate on  square arrays with empty cells  where every row sum and every column sum is $0$ in $\Z$.  We solve most of the instances of this case.

\end{abstract}



\section{Introduction and some examples}

A Heffter array $H(m,n;s,t)$ is an $m \times n$ matrix with nonzero entries from $\mathbb{Z}_{2ms+1}$ such that 
\begin{enumerate}
  \item 
each row contains $s$ filled cells and each column contains $t$ filled cells,
\item the elements in every row and column sum to 0 in $\mathbb{Z}_{2ms+1}$, and 
\item  for every $x \in \mathbb{Z}_{2ms+1} \setminus \{0\}$, either $x$ or $-x$ appears in the array.
\end{enumerate}

The notion of a Heffter array $H(m,n;s,t)$ was first defined by Archdeacon in \cite{heffter1}.  It is shown there that a Heffter array with a pair of special orderings can be used to construct an embedding of the complete graph $K_{2ms+1}$ on a surface. The connection is given in the following theorem.

\begin{theorem}\label{heffter1}{\rm \cite{heffter1}} Given a Heffter array $H(m,n;s,t)$ with compatible orderings $\omega_r$ of the symbols in the rows of the array and $\omega_c$ on the symbols in the columns of the array, then there exists an embedding of $K_{2ms+1}$ such that every edge is on a face of size $s$ and a face of size $t$. Moreover, if $\omega_r$ and $\omega_c$ are both simple, then all faces are simple cycles. \end{theorem}

We will not concern ourselves with the ordering problem in this paper and will just concentrate on the construction of the Heffter arrays. (In a subsequent paper (see \cite{heffter4}) we will address the ordering problem in more detail.)  
We refer the reader to \cite{heffter1} for the definition of a simple ordering and the definition of compatible orderings.


 Theorem \ref{heffter1} can be stated in design theoretic terms.  A $t-$cycle system on $n$ points is a decomposition of the edges of $K_n$ into $t-$cycles.  A $t-$cycle system $C$ on  $K_n$ is {\em cyclic} if there is a labeling of the vertex set of $K_n$ with the elements of $\Z_n$ such that the permutation  $x\rightarrow x+1$ preserves the cycles of $C$.
A biembedding of an $s-$cycle system  and a $t-$cycle system is a face 2-colorable topological
embedding of the complete graphs $K_{2ms+1}$ in which one color class is comprised of the cycles in the $s-$cycle system and the other class contains the cycles in the $t-$cycle system.
We refer the reader to \cite{handbook-cycle} for general information on $t-$cycle systems, to \cite{BD,V} for cyclic $t-$cycle systems and to \cite{GG} for information about biembedings of cycle systems. While not using Heffter arrays, \cite{DGLM} and \cite{DGLM2} gives a related idea of embedding a single 2-fold triple system.
The following proposition about cycle systems follows from Theorem \ref{heffter1}.

\begin{proposition}\label{heffter-cycle} Assume there exists a  Heffter array $H(m,n;s,t)$ with compatible orderings $\omega_r$ of the symbols in the rows of the array and $\omega_c$ on the symbols in the columns of the array such that $\omega_r$ and $\omega_c$ are both simple.  Then there exists  a biembedding of a cyclic $s-$cycle system $S$ and a  cyclic $t-$cycle system $T$ both  on $2ms+1$ points.  
\end{proposition}

A Heffter array is called an {\em integer} Heffter array if Condition 2 in the definition of Heffter array above is strengthened so that the elements in every row and every column sum to zero in $\Z$.  In this paper we will concentrate on constructing {\em square integer} Heffter arrays. If the Heffter array is square, then $m=n$ and necessarily $s=t$. So for the remainder of this paper define a Heffter array $H(n;k)$ to be an $n \times n$ array of integers  satisfying the following:

\begin{enumerate}
  \item  each row  and each column contains $k$ filled cells,
  \item the symbols in every row and every column sum to 0 in $\Z$, and
  \item for every element $x\in\{1,2, \ldots ,nk\}$ either $x$ or $-x$ appears in the array.
\end{enumerate}

In  \cite{heffter2} the authors study the case where the Heffter array has no empty cells.  In that paper it is shown that there is an {\em integer} $H(m,n;n,m)$ if and only if $m,n \geq 3$ and $mn \equiv 0,3$ (mod 4) and in general that there is an  $H(m,n;n,m)$ for all $m,n \geq 3$.   A {\em shiftable} Heffter array $H_s(m,n;s,t)$ is defined to be a Heffter array $H(m,n;s,t)$ where every row and every column contain the same number of positive and negative numbers.   From  \cite{heffter2} we have that there is an  $H_s(m,n;n,m)$ if and only if $m$ and $n$ are even. The notation $H_s(n;2k)$ denotes a {\em shiftable} $H(n;2k)$.

In this paper we extend the idea of shiftable to any array.  An $ n \times n$ array  $A$ of integers (possibly with empty cells) is {\em shiftable} if each row and each column contains the same number of positive and negative numbers.   Let  $A$ be a shiftable array and $x$ a nonnegative integer. If $x$ is added to each  positive element and $-x$ is added to each negative element, then all of the row and column sums remain unchanged. Let $A \pm x$ denote the array  where $x$ is added to all the {\em positive} entries in $A$ and $-x$ is added to all the {\em negative} entries.  

If $A$ is an integer array, define the {\em support} of $A$ as the set containing the absolute value of the elements contained in $A$.  So if $A$ is shiftable with support $S$ and $x$ a nonnegative integer, then $A\pm x$ has the same row and column sums as $A$ and has support $S + x$.   In the case of a shiftable Heffter array $H_s(n;k)$, the array 
$H_s(n;k)\pm x$ will have row and column sums equal to zero and support $S= \{1+x,2+x, \ldots ,nk+x\} $.

The following lemma gives necessary conditions for the existence of  $H(n;k)$ and  $H_s(n;k)$.

\begin {lemma}\label{necessary}  If there exists an $H(n;k)$, then necessarily $nk \equiv 0,3$ (mod 4).  Furthermore,
if there exists an $H_s(n;k)$, then necessarily $k$ is even and $nk \equiv 0$ (mod 4).  \end{lemma}  

\begin{proof}  Given an $H(n;k)$, in order for each row to sum to zero, each row must contain an even number of odd numbers.  Hence the entire array contains an even number of odd numbers.   Now, the support of  the $H(n;k)$ is the  set $S =\{1,2,  \ldots ,nk\}$.  There will be an even number of odd numbers in $S$ exactly when $nk \equiv 0,3$ (mod 4).

If there is an $H_s(n;k)$, then clearly $k$ must be even in order to have the same number of positive and negative entries in each row.  It follows that $nk \not\equiv 3$ (mod 4), hence $nk \equiv 0$ (mod 4). 
\end{proof}

We give some examples of Heffter arrays $H(n;k)$ and $H_s(n;k)$. 

\begin{example} \label{ex1} The following are $H(4;3), H_s(5;4), H_s(6;4)$, $H_s(7;4)$ and $H_s(8;6)$ respectively.
\end{example}

\renewcommand{\tabcolsep}{3pt}

\begin{center}
\begin{tabular}{|c|c|c|c|}  \hline
4& 8 & & -12\\ \hline
-9 & 3 & 6 &  \\ \hline
 & -11 & 1 & 10 \\ \hline
5 &  & -7 & 2 \\ \hline
\multicolumn{4}{c}{$H(4;3)$}\\
\end{tabular} \hspace{.3in}
\begin{tabular}{|c|c|c|c|c|} \hline
 & 17 & -8 & -14&5  \\ \hline
1& & 18& -9 &-10\\ \hline
-6 & 2 &  &19 &-15\\ \hline
-11&-12&3&&20 \\ \hline
16&-7&-13&4& \\ \hline
\multicolumn{5}{c}{$H_s(5;4)$}\\
\end{tabular} \hspace{.3in}
\begin{tabular}{|c|c|c|c|c|c|} \hline
&&13&-15&-9&11\\ \hline
&&-14&16&10&-12 \\ \hline
-1&3&&&17&-19\\  \hline
2&-4&&&-18&20 \\ \hline
21&-23&-5&7&& \\ \hline
-22&24&6&-8&& \\
\hline
\multicolumn{6}{c}{$H_s(6;4)$}\\
\end{tabular} 

\bigskip
\begin{tabular}{|c|c|c|c|c|c|c|} \hline
1&&&&26&-13&-14 \\ \hline
-8&2&&&&27&-21 \\ \hline
-15&-16&3&&&&28 \\ \hline
22&-9&-17&4&&& \\ \hline
&23&-10&-18&5&& \\ \hline
&&24&-11&-19&6& \\ \hline
&&&25&-12&-20&7 \\ \hline
\multicolumn{7}{c}{$H_s(7;4)$}\\
\end{tabular}\hspace{.6in}
\begin{tabular}{|c|c|c|c|c|c|c|c|} \hline
-1&5&2&-7&-9&10&&  \\  \hline
3&-4&-6&8&11&-12&&    \\ \hline
&&-13&17&14&-19&-21&22    \\ \hline
&&15&-16&-18&20&23&-24    \\ \hline
-33&34&&&-25&29&26&-31    \\ \hline
35&-36&&&27&-28&-30&32    \\ \hline
38&-43&-45&46&&&-37&41    \\ \hline
-42&44&47&-48&&&39&-40    \\ \hline
\multicolumn{8}{c}{$H_s(8;6)$}\\
\end{tabular}  

\end{center}
\renewcommand{\tabcolsep}{6pt}

In this paper we will prove the existence of Heffter arrays $H(n;k)$ for many of the possible values of $n$ and $k$.    The paper is organized as follows.  In Section \ref{section2} we cover the case when $k$ is even and in Section \ref{section3} we deal with the case when  $k \equiv 3$ (mod 4).  In both of these cases we prove that the necessary conditions from   Lemma \ref{necessary} are sufficient.  In 
Section \ref{section4} we consider the case when  $k \equiv 1$ (mod 4).  In this case we do not get a complete solution, however we construct $H(n;k)$ for many values of $n$ and $k$.  The results are summarized in Section \ref{section5}.

\section{An even number of filled cells per row and column} \label{section2}

We give two direct constructions which both yield shiftable Heffter arrays with an even number $k$ of cells per row and column. From Lemma \ref{necessary},  if $k\equiv 2$ (mod 4), then $n$ must be even, while if $k\equiv 0$ (mod 4), then any $n\geq 4$ is possible. We will construct Heffter arrays in both of these cases.   The first construction covers the case when $n$ is even and the second one covers the case when $n$ is odd. 

\begin{theorem} \label{th1} There exists an $H_s(n;k)$ for all $n,k$ even with   $n \ge k \geq 4$.  
\end{theorem}

\begin{proof}

For $ i \geq 0$ let $A_i$, $B_i$ and $C_i$ be  $2 \times 2$ arrays defined by

\renewcommand{\tabcolsep}{2pt}
\begin{center}

$A_i = $ \begin{tabular}{|c|c|} \hline
$-1-4i$&$2+4i$\\ \hline
$3+4i$& $-4-4i$ \\ \hline
\end{tabular}\hspace{.3in}
$B_i =$ \begin{tabular}{|c|c|} \hline
$1+4i$&$-2-4i$\\ \hline
$-3-4i$& $4+4i $\\ \hline
\end{tabular}\hspace{.3in}
$C_i =$ \begin{tabular}{|c|c|} \hline
$1+4i$&$-3-4i$\\ \hline
$-2-4i$&$ 4+4i $\\ \hline
\end{tabular}

\end{center}

\medskip

A few things to note first.  If {\em any}  $A_i$ and  $B_j$ are aligned in the same two rows, they contribute zero to both of the row sums. Similarly if they are aligned in any two columns, they contribute zero to both column sums.  Also if $A_i$, $A_j$ and $C_k$ (for any $i,j,k$) are aligned in the same two rows, they contribute zero to both of the row sums.  However if they are aligned in any two columns, they contribute $+3$ to the first column and $-3$ to the second.

We first present the proof in the case when $k \equiv 0$ (mod 4). We will construct an $H_s(2n;4k)$ when $k \leq  n/2 $.
Let $D$ be a $n \times n$ empty array. In $D$ find two disjoint sets of $k$ disjoint transversals, $T_1$ and T$_2$. This is easy to do by just choosing $2k$ (broken) diagonals and letting the first $k$ of them be $T_1$ and the remaining $k$ of them be  $T_2$ (but any two sets of $k$ disjoint transversals will work). Note that this implies that $n \ge 2k$.   Now in the $nk$ total cells of the $T_1$ transversals place the arrays $A_0, A_1, \ldots ,A_{nk-1}.$  In the cells of  the $T_2$ transversals place the arrays $B_{nk}, B_{2nk+1}, \ldots ,B_{2nk-1}$.  Denote the resulting array as $H_D$. Since each row and column of $H_D$ contains the same number of $A_i$ squares as $B_i$ squares, by the paragraph above  all the row and column sums are equal to zero.  Also it is easy to see that each row and column of the resulting array contains $2k$ positive numbers and $2k$ negative numbers.  Hence we have constructed an  $H_s(2n;4k)$.

Next we deal with the case when $k \equiv 2$ (mod 4).  We will construct an $H_s(2n;k)$ where $k=2r$ with $r$ odd and $r \leq n$.  Let $D$ be an empty $n \times n$ array where the rows and columns are indexed by $0,1,\ldots ,n-1$. In cells $(i,i), (i,i+1)$ and $(i,i+2)$, $0\leq i \leq n-1$ of $D$ place the arrays $A_{3i},C_{3i+1}$ and   $A_{3i+2}$, respectively where the row arithmetic is performed in $\Z_n$. Again, denote the resulting array as $H_D$. As noted above, all the row sums of $H_D$ will be zero, while the column sums will be $3,-3,3,-3,\dots, -3$.  Now, let $e<n$ be an even number.  In row $e$ we note that the upper right cell of  $A_{3e}$ is $2+12e$, while the upper left corner of $C_{3e+1}$ is $1+4(3e+1) = 5+12e$.  Hence the upper left cell of $C_{3e+1}$ is three more than the upper right cell of  
$A_{3e}$.  Swapping these two cells in every even row will reduce each even column of $H_D$ by three while increasing each odd column by 3. The result is that now each column also has sum zero.    Clearly, each row and column contains the same number of positive and negative values, hence we have constructed an 
$H_s(2n;6)$ for all $n \geq 6$.  To complete the proof all that needs to be done is to add $r-3$ paired transversals of A's and B's as was done in the prior paragraph (use $(r-3)/2$ transversals for the A's and $(r-3)/2$ for the B's).  This yields an $H_s(2n;2r)$ in the case where $r$ is odd, completing the proof.
\end{proof}

The $H_s(6;4)$ and the $H_s(8;6)$  given in Example \ref {ex1} were constructed via the method of Theorem \ref{th1}.  


\begin{theorem} \label{th2} There exists an $H_s(n;4)$ for all $n \ge 4$.  
\end{theorem}

\begin{proof} Label the rows and columns $1,2, \ldots ,n$.
Place the symbols $1, -(n+1),-(2n+1), 3n+1$ in the first column starting in rows 1,2,3 and 4 respectively.  
For each column $c$ with $2 \leq c \leq n-1$, place the symbols $c, -(2n+c), -(n+c), 3n+c$ in  rows
$c,c+1,c+2,c+3$, respectively (arithmetic on the rows is modulo $n$).  Finally, in the last column place $n,-2n,-3n,4n$ in rows $n,1,2,3$, respectively.  

It is easy to see that the support of this array is $\{1,2, \ldots, 4n\}$ and that each row and column contains 2 positive values and 2 negative values.  Also since $c-(n+c)-(2n+c)+ (3n+c)=0$, we see that each column sum is zero as desired.  Now we check the row sums. We check the first three rows individually. Row 1 contains the symbols $1, 4n-2, -(2n-1),$ and $-(2n)$  so the sum is zero.  Row 2 contains
$-(n+1), 2, (4n-1)$ and $-(3n)$ also adding to zero. Row 3 contains $-(2n+1), -(2n+2), 3, 4n$, again adding to zero.  Now let $4\leq r \leq n$.  The symbols in row $r$ are $r, -(2n+r-1), -(n+r-2), 3n+r-3$  (working backwards from the diagonal entry).  Since $r -(2n+r-1) -(n+r-2) + (3n+r-3)=0$, we have that each row adds to zero and hence we have constructed an $H_s(n;4)$.
\end{proof}

The $H_s(7;4)$ given in Example \ref {ex1} was constructed via the method of Theorem \ref{th2}. 


\medskip

If $D$ is an $n \times n$ array with rows and columns labeled $0,1,\ldots, n-1$, for $i = 0,1 \ldots, n-1$ define the $i$th  diagonal $D_i$ to be the set of cells $D_i= \{(0,i),(1,i+1), \ldots,(n-1,i-1)\}$ where all arithmetic is performed in $\Z_n$.  We say that the diagonals $D_i$ and $D_{i+1}$ are {\em consecutive} diagonals.
We see that the $H_s(n;4)$ constructed via Theorem \ref{th2} has the property that all of the filled cells are contained in exactly four consecutive diagonals of the array.  We now show how to use this fact to add four filled cells per row and column to an existing Heffter array $H$ if $H$ contains four consecutive diagonals of empty cells.  

\begin{lemma} \label{add4} If there exists an integer Heffter array $H(n;k)$ which  has $s$ disjoint sets of four consecutive empty diagonals,  then there exists an
$H(n;k+4s)$.  Furthermore, if the $H(n;k)$ is shiftable, then there exists an $H_s(n;k+4s)$.
\end{lemma}

\begin{proof}
Let $H$ be an integer Heffter array $H(n;k)$ which has four consecutive empty diagonals, say $D_i, D_{i+1}, D_{i+2}$ and $D_{i+3}$.  Let $J$ be an $H_s(n;4)$ constructed via Theorem \ref{th2}. Note that in row $i+3$ of $J$, the filled cells are in columns $i,i+1,i+2$ and $i+3$. Cyclically permute the rows of $J$ so that row $r$ moves to row $r-i-3$ (mod $n$).  Note this places row $i+3$ as the new row 0 and that the filled cells of $J$ are now contained in the diagonals $D_i, D_{i+1}, D_{i+2}$ and $D_{i+3}$. 
Clearly $J$ is a shiftable array.  Let $J' = J\pm nk$.  As noted before, $J'$ has row and column sums equal zero and support 
$S= \{1+nk,2+nk, \ldots ,x+ nk\}$.

Now combine (add) $H$ and $J'$.  This array now contains $k+4$ filled cells in each row and each column and furthermore the filled cells in each row and each column add to $0$. Hence we have constructed an
 $H(n;k+4)$, or an $H_s(n;k+4)$ if the $H(n;k)$ was shiftable.  If the $H(n;k)$ has $s$ disjoint sets of four consecutive empty diagonals, then $s$ repeated applications of this process yields an $H(n;k+4s)$ whenever $k+4s \leq n$.  \end{proof}

As a corollary we get the existence of Heffter arrays $H_s(n;4k)$ for all $4\leq 4k \leq n$.

\begin{corollary}\label{4k}
There exists an $H_s(n;4k)$ for all $n$ and $4\leq 4k \leq n$.  
\end{corollary}

\begin{proof} Begin with the  $H_s(n;4)$ from Theorem \ref{th2} and apply Lemma \ref{add4}.  \end{proof}

We summarize the main results of this section in the next theorem.

\begin{theorem}There exists an $H_s(n;k)$ if and only if $k$ is even and $nk \equiv 0$ (mod 4). 
\end{theorem}

\begin{proof}  This follows from Lemma \ref{necessary}, Theorem \ref{th1} and Corollary \ref{4k}. \end{proof}

\section{H({\em n};{\em k}) with {\em k} $\equiv$ 3  (mod 4) } \label{section3}

The first person to recognize the relation between combinatorial designs and graph embeddings was Heffter \cite{H}, who showed how to describe embeddings combinatorially using a solution to a difference problem in modulo arithmetic. He used this to construct triangular biembeddings of some complete graphs. It is still unknown if his construction yields an infinite class. 

While proving the Map Color Theorem \cite{R} Ringel and Youngs showed how to record these edge labelings as a type of flow on a cubic graph called a current graph (a precise definition will follow shortly). At first current graphs were considered a kind of nomogram of little interest independently; however, in 1974 Gross and Alpert \cite{GA} developed a general theory of current graphs. In \cite{Y} Youngs gave current graphs based on M\"{o}bius ladders with $n=4m+1$ rungs and on cylindrical ladders with $n=4m$ rungs (yielding infinite classes of graph embeddings). A further discussion of Young's current assignment on ladder graphs appears in \cite{anderson}. We discuss these constructions and how they give our desired Heffter arrays $H(n;3)$. We then extend the construction to $H(n;k)$ for all $k \equiv 3 \pmod 4$ with $3\leq k< n$. 

An {\em arc} $\vec e$ in a graph $G$ is edge $e = \{u,v\}$ directed in one of two ways: $(u,v)$ or $(v,u)$. The first vertex in an arc is the {\em tail}, the second the {\em head}. The set of arcs in $G$ is denoted $A(G)$. If $\vec e$ is an arc, let $-\vec e$ be the same edge with opposite direction. Let $S$ be the set of $|A(G)|$ integers $\{ \pm 1,\dots,\pm |E(G)|\}$. 

\medskip

An {\em integer-current assignment} is a bijection $\kappa : A(G) \rightarrow S$ such that
\begin{enumerate}
\item {\em (respects negatives)} $\kappa(-\vec e) = - \kappa(\vec e)$, and
\item {\em (Kirchoff's current law, KCL)} for each $u \in V(G)$, $\kappa(\vec e_1) + \cdots + \kappa(\vec e_k) = 0$ where $\{\vec e_1,\dots, \vec e_k\}$ are the set of arcs with tail $u$. 
\end{enumerate}

Bipartite current graphs with an integer-current assignment are closely related to Heffter arrays. 

\begin{lemma}\label{equivalence} There exists a $k$-regular bipartite graph of order $2n$ with an integer-current assignment if and only if there exists an $H(n;k)$.  \end{lemma}

\begin{proof} Let $R\, \cup\, C$ be a bipartition of the vertices in the given graph $G$. Let $A$ be an $n \times n$ array whose rows are indexed by $R$ and columns by $C$. In row $i$, column $j$ place $a_{i,j} = \kappa((i,j))$. All entries are distinct up to sign and {\em KCL} shows that the row and column sums of $A$ are all 0. The construction is easily reversed to build the graph $G$ from the array $A$. \end{proof} 

From Lemma \ref{necessary} if there exists an $H(n;k)$, then $nk \equiv 0,3$ (mod 4). In this section  $k \equiv 3$ (mod 4), so $n\equiv 1$ or 0 (mod 4). Section \ref{sect3.1} studies $n\equiv 1$ (mod 4) while Section \ref{sect3.2} studies $n\equiv 0$ (mod 4). 

\subsection{{\em H(n;k)} with  {\em n} $\equiv$ 1 (mod 4) and {\em k} $\equiv$ 3 (mod 4)} \label{sect3.1}

A {\em M\"{o}bius ladder} on $4m+1$ {\em rungs} is a bipartite graph with vertex set $R \, \cup \, C$ where $R = \{r_i \,|\, 1 \le i \le 4m+1\}$ and $C = \{c_i \,|\, 1 \le i \le 4m+1\}$ 
and edge set 
$\{ \{r_i,c_{j}\} \, | \, 1 \le i \le 4m+1, j = i-1, j=i \mbox{ or } j=i+1\}$ 
(the subscripts are read modulo $4m+1$). Example \ref{m13} shows the M\"{o}bius ladder on 13 rungs. 

Youngs \cite{Y} gives the following integer-current assignment to these graphs. He considered the entries as elements of the integers modulo $24 m + 7$, but interestingly noted it has the ``further aesthetic advantage'' that KCL holds over the integers as well. We have verified this claim.

\newpage
\begin{tabl} \label {3.2} {\rm \cite{Y}} Currents on the M\"{o}bius ladder with $4m+1$ rungs. \end{tabl}
\begin{center}
$\begin{array}{lll}
\kappa((r_{2i-1},c_{2i})) = 8m+3-i, & \kappa((c_{2i-1},r_{2i})) = 8m+2+i, & i=1,\dots,m\\
\kappa((r_{2i},c_{2i+1})) = 12m+3-i,& \kappa((c_{2i},r_{2i+1})) = 4m+2+i & i=1,\dots,m\\
\kappa((r_{2m+2i-1},c_{2m+2i})) =9m+2+i, & \kappa((c_{2m+2i-1},r_{2m+2i})) = 7m+3-i & i=1,\dots,m\\
\kappa((r_{2m+2i},c_{2m+2i+1})) = 5m+2+i, &\kappa((c_{2m+2i},r_{2m+2i+1})) = 11m+3-i  & i=1,\dots,m\\
\kappa((c_i,r_i)) = 4m+1-i & & i=1,\dots,2m \\
\kappa((r_{2m+i+1},c_{2m+i+1})) = 2m-i &  &i = 1,\dots,2m-1 \\
\kappa((r_{4m+1},c_{1})) = 12m+3 & \kappa((c_{4m+1},r_{1})) = 4m+2 \\
\kappa((c_{2m+1},r_{2m+1})) = 4m+1 & \kappa((c_{4m+1},r_{4m+1})) = 2m
\end{array}$
\end{center}

In the next examples we give an  integer-current assignment for the M\"{o}bius ladder on 13 rungs followed by the resulting $H(13;3)$.

\begin{example}\label{m13} An integer-current assignment for the M\"{o}bius ladder on 13 rungs ($m = 3$). 
\end{example}

\begin{figure}[h] 
  \centering
  
  \includegraphics[bb=0 -1 579 105,width=6.30in,keepaspectratio]{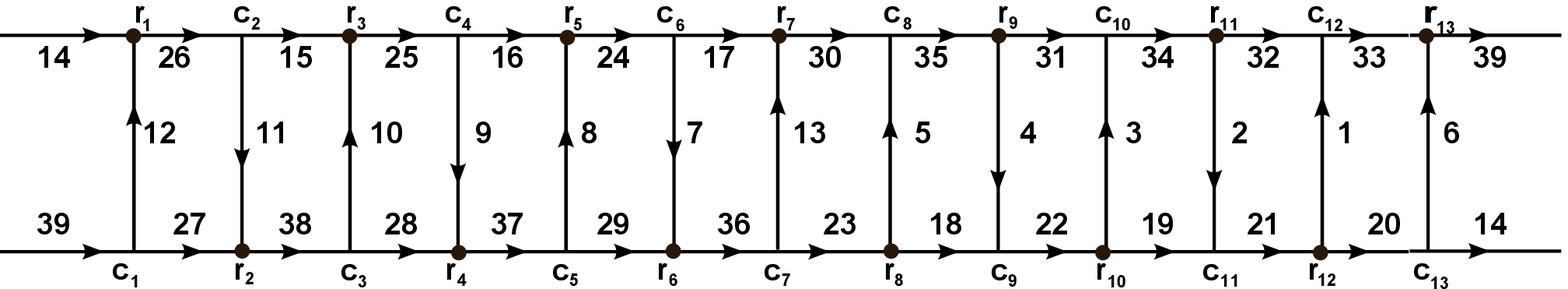}

  \label{fig:ladder13rung}
\end{figure}

\renewcommand{\tabcolsep}{2pt}
\begin{example} The $H(13;3)$ resulting from Example \ref{m13} using Lemma \ref{equivalence}.
\label{h13}
\end{example} {\small
\begin{center}
\begin{tabular}{|c|c|c|c|c|c|c|c|c|c|c|c|c|} \hline
-12&26&&&&&&&&&&&-14 \\ \hline
-27&-11&38&&&&&&&&&& \\ \hline
&-15&-10&25&&&&&&&&& \\ \hline
&&-28&-9&37&&&&&&&& \\ \hline
&&&-16&-8&24&&&&&&& \\ \hline
&&&&-29&-7&36&&&&&& \\ \hline
&&&&&-17&-13&30&&&&& \\ \hline
&&&&&&-23&5&18&&&& \\ \hline
&&&&&&&-35&4&31&&& \\ \hline
&&&&&&&&-22&3&19&& \\ \hline
&&&&&&&&&-34&2&32& \\ \hline
&&&&&&&&&&-21&1&20 \\ \hline
39&&&&&&&&&&&-33&-6 \\ \hline
\end{tabular}
\end{center}
} 

An array $A = A(i,j)$ is {\em cyclically tridiagonal} if all nonzero entries $A(i,j)$ have $|i - j| \leq 1$ except for $A(1,n)$ and $A(n,1)$. The $H(13;3)$ in Example \ref{h13} above is cyclically tridiagonal.
Using Lemma \ref{equivalence} and Youngs' current graph we get the following. 

\begin{theorem} \label{4m+1,3}  There exists a cyclically tridiagonal $H(4m+1;3)$ for all $m \geq 1$.  
\end{theorem}

This serves as the base case for the main theorem of this subsection.

\begin{theorem} \label{ladder.4m+1}There exists an $H(n;k)$  for every $n \equiv 1$ (mod 4) (with $n \geq 5$) and every 
$k \equiv 3$ (mod 4) with $3 \leq k <n$.\end{theorem}

\begin{proof} Assume $n \geq 5$ with $n \equiv 1$ (mod 4) and  $3 \leq k \leq n$.  Write 
$k=4s+3$. From Theorem \ref{4m+1,3}, there exists a cyclically tridiagonal $H(n;3)$. Apply Lemma \ref{add4} to obtain the desired $H(n;4s+3)$.
\end{proof}

In Section \ref{section4} we will use the Heffter array constructed in Theorem \ref{4m+1,3}  as an ingredient in a construction of a $H(n;5)$ for certain values of $n$.  To do so we need a property of our constructed $H(4m+1;3)$.

A {\em transversal} $T$ in an $n\times n$ array $X=X(r,c)$ is a set of $n$ non-empty cells $\{(r_1,c_1), \dots,$ $ (r_n,c_n)\}$ such that whenever $i \ne j$ we have $r_i\neq r_j$, $c_i\neq c_j$, and $X(r_i,c_i)\neq X(r_j,c_j)$. A transversal in a Heffter array $H(n;k)$ is {\em primary} if $\{|X(r_i,c_i)|\mid 1\leq i\leq n\}=\{1,\dots, n\}$.
A $H(n;k)$ array ${\cal H}$ is  {\em strippable} if there exists a primary transversal $T$  in ${\cal H}$ such that ${\cal H}\setminus T$ is shiftable; that is,  each row and  each column of ${\cal H}\setminus T$ contains $(k-1)/2$  positive integers and  $(k-1)/2$ negative integers.  

\begin{corollary} \label{4n+1.strip}There exists a strippable $H(4m+1;3)$ for all $m \geq 1$.  
\end{corollary}

\begin{proof} The $H(4m+1;3)$ constructed from Theorem \ref{4m+1,3} has the main diagonal as a primary transversal. There are two remaining broken diagonals: one filled with positive numbers and the other with negative. Hence the array is strippable. \end{proof}

\subsection{{\em H(n;k)} with {\em n} $\equiv$ 0 (mod 4) and {\em k} $\equiv$ 3 (mod 4)} \label{sect3.2}

The proof in this case is analogous to that of the previous section. We again use an integer-current graph.

A {\em cylindrical ladder} on $4m$ {\em rungs} is a bipartite graph with vertex set $R \, \cup \, C$ where $R = \{r_i \,|\, 1 \le i \le 4m\}$ and $C = \{c_i \,|\, 1 \le i \le 4m\}$ and edge set $\{ \{r_i,c_{j}\} \, | \, 1 \le i \le 4m, j=i-1, j=i, \mbox{ or } j=i+1   \}$ (the subscripts are read modulo $4m$). Example \ref{c12} shows the cylindrical ladder on 12 rungs. 

Youngs \cite{Y} gives the following integer-current assignment to these graphs. He was interested in the entries as elements of the integers modulo $24 m + 1$, but again noted {\em KCL} holds over the integers. We have again verified this claim.

\newpage
\begin{tabl} \label {3.8} {\rm \cite{Y} }Currents on the cylindrical ladder with $4m$ rungs. \end{tabl}
\begin{center}
$\begin{array}{lll}
\kappa((r_{2i-1},c_{2i})) = 8m+1-i, & \kappa((c_{2i-1},r_{2i})) = 8m+i & i=1,\dots,m\\
\kappa((r_{2i},c_{2i+1})) = 12m-i, & \kappa((c_{2i},r_{2i+1})) =4m+1+i & i=1,\dots,m-1\\
\kappa((r_{2m-2+2i},c_{2m-1+2i})) = 5m+i, & \kappa((c_{2m-2+2i},r_{2m-1+2i})) = 11m+1-i & i=1,\dots,m\\
\kappa((r_{2m-1+2i},c_{2m+2i})) = 9m+i, & \kappa((c_{2m-1+2i},r_{2m+2i})) = 7m+1-i & i=1,\dots,m\\
\kappa((c_{i+1},r_{i+1})) = 4m-1-i &  & i=1,\dots,2m-2 \\
\kappa((r_{2m+i},c_{2m+i})) = 2m-i  && i = 1,\dots,2m-1 \\
\kappa((r_{4m},c_{1})) = 4m+1 & \kappa((c_{4m},r_{1})) = 12m \\
\kappa((r_{1},c_{1})) = 4m & \kappa((r_{2m},c_{2m})) = 4m-1\\
\kappa((r_{4m},c_{4m})) = 2m
\end{array}$
\end{center}

In Example \ref{ex1} we presented an $H(4;3)$ constructed from the cylindrical ladder on 4 rungs using the equivalence from Lemma \ref{equivalence}.  In the next examples  we give an  integer-current assignment for the cylindrical ladder on 12 rungs followed by the resulting $H(12;3)$.

\begin{example}\label{c12} An integer-current assignment for the cylindrical ladder on 12 rungs ($m = 3$). 
\end{example}


\begin{figure}[h] 
  \centering
 \includegraphics[bb=0 -1 522 104,width=6.5in,keepaspectratio]{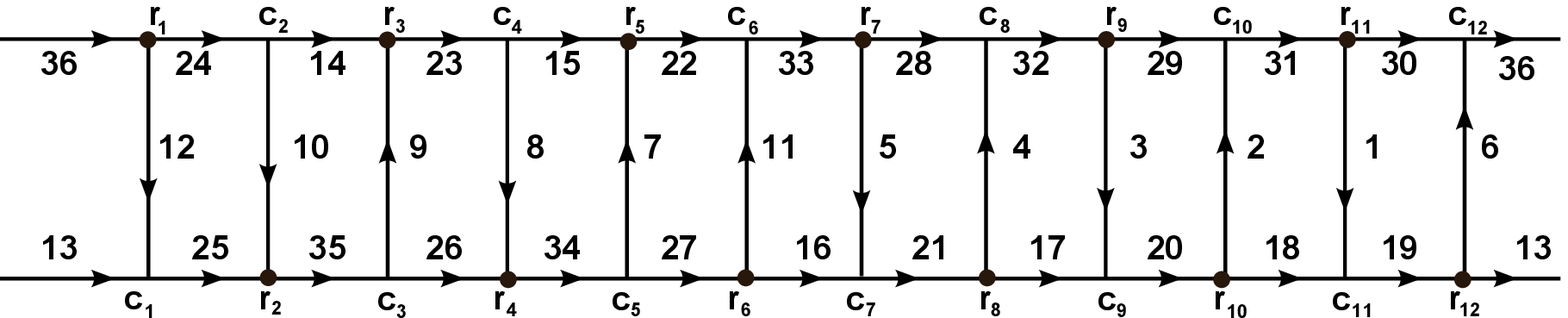}

  \label{fig:ladder12rung}
\end{figure}

\renewcommand{\tabcolsep}{2pt}
\begin{example} The $H(12;3)$ resulting from Example \ref{c12} using Lemma \ref{equivalence}. 
\label{h12}
\end{example} {\small
\begin{center}
\begin{tabular}{|c|c|c|c|c|c|c|c|c|c|c|c|} \hline
12&24&&&&&&&&&&-36 \\ \hline
-25&-10&35&&&&&&&&& \\ \hline
&-14&-9&23&&&&&&&& \\ \hline
&&-26&-8&34&&&&&&& \\ \hline
&&&-15&-7&22&&&&&& \\ \hline
&&&&-27&11&16&&&&& \\ \hline
&&&&&-33&5&28&&&& \\ \hline
&&&&&&-21&4&17&&& \\ \hline
&&&&&&&-32&3&29&& \\ \hline
&&&&&&&&-20&2&18& \\ \hline
&&&&&&&&&-31&1&30 \\ \hline
13&&&&&&&&&&-19&6 \\ \hline
\end{tabular}
\end{center}

} 

Using Lemma \ref{equivalence} and the current graphs above we get the following. 

\begin{theorem} \label{4m,3}  There exists a cyclically tridiagonal $H(4m;3)$ for all $m \geq 1$.  
\end{theorem}

It is again apparent that the  $H(4m;3)$ constructed in Theorem \ref{4m,3} is strippable.  

\begin{corollary}\label{strip-4m,3} There exists a strippable $H(4m;3)$ for all $m \geq 1$. \end{corollary}

The proof that this result extends to the case $k \equiv 3 \pmod 4$ is exactly the same as in 
Theorem \ref{ladder.4m+1} using the fact that the $H(4m;3)$ is cyclically tridiagonal. 

\begin{theorem} \label{ladder.4m}There exists an $H(n;k)$  for every $n \equiv 0$ (mod 4) and every 
$k \equiv 3$ (mod 4) with $3 \leq k < n$.\end{theorem}

We summarize the results of Section \ref{section3} for future reference.

\begin{theorem} \label{k=3mod4} There exists an $H(n;k)$ with $k \equiv 3$ (mod 4) if and only if  $n \geq 4$ and $n \equiv 0,1$ (mod 4). \end{theorem}

\section {H({\em n};{\em k}) with {\em k} $\equiv$ 1 (mod 4)}  \label{section4}

By Lemma \ref{necessary}, in order for an $H(n;k)$ with $k \equiv 1$ (mod 4) to exist it  is necessary that $n\equiv 0$ or 3 modulo 4. We consider these two cases in two subsections.

\subsection{{\em H(n;k)} with {\em n} $\equiv$ 0 (mod 4) and {\em k} $\equiv$ 1 (mod 4)} 
\label{subsection4.1}

The smallest example in this case is an $H(8;5)$.  It is displayed below.

\begin{example} An $H(8;5)$.
\end{example}
\begin{center}
\begin{tabular}{|r|r|r|r|r|r|r|r|} \hline
13&$-$20&19&$-$11&$-$1&&& \\ \hline
$-$14&16&18&$-$22&&2&& \\ \hline
15&$-$12&$-$23&17&&&3& \\ \hline
$-$9&10&$-$21&24&&&&$-$4 \\ \hline
$-$5&&&&29&$-$30&31&$-$25 \\ \hline
&6&&&$-$36&32&$-$28&26 \\ \hline
&&7&&35&34&$-$39&$-$37 \\ \hline
&&&$-$8&$-$27&$-$38&33&40 \\ \hline
\end{tabular}
\end{center}
\medskip

For our next result we will use a collection of $4 \times 4$ arrays called $B(a,b)$ which satisfy the following properties:

\renewcommand{\tabcolsep}{6pt} 

\begin{tabular}{ll}
1.&the support of $B(a,b)$  is  $\{ 1,2, \ldots ,16 \}$,\\
2. &the sum of the elements in row 1 and row 4  is $a$,\\
3.& the sum of the elements in row 2 and row 3  is $-a$,\\
4.& the sum of the elements in column 1 and column 4  is $b$,\\
5.& the sum of the elements in column 2 and column 3  is $-b$,\\
6.& each row and each column contains  exactly two positive and two negative entries \\
& (i.e. $B(a,b)$ is shiftable). \\
 
\end{tabular}

\renewcommand{\tabcolsep}{2pt} 

 \begin{lemma} \label{boosters} There exist arrays $B(a,b)$ for  all $a,b \in \{0,4,8,12\}$.
\end{lemma}

\begin{proof}  We will give each of these arrays explicitly. First note that if a $B(a,b)$ exists, then its transpose is a $B(b,a)$.  Hence below we will list the necessary arrays when $ a\leq b$.

\renewcommand{\arraycolsep}{2pt}
$$
\begin{array}{|c|c|c|c|}
\multicolumn{4}{c} {B(0,0)} \\
\hline
1&-2&-3&4   \\ \hline
-5&6&7&-8\\ \hline
-9&10&11&-12\\ \hline
13&-14&-15&16\\ \hline
 \end {array} \hspace{.2in}
\begin{array}{|c|c|c|c|}
\multicolumn{4}{c} {B(0,4)} \\
\hline
-1&5&9&-13\\ \hline
-2&6&10&-14\\ \hline
3&-7&-11&15\\ \hline
4&-8&-12&16\\ \hline
\end{array} \hspace{.2in}
\begin{array}{|c|c|c|c|}
\multicolumn{4}{c} {B(0,8)} \\
\hline
-1&2&3&-4\\ \hline
5&-6&-7&8\\ \hline
-9&10&11&-12\\ \hline
13&-14&-15&16\\ \hline
\end{array}\hspace{.2in}
\begin{array}{|c|c|c|c|}
\multicolumn{4}{c} {B(0,12)} \\
\hline
-2&4&6&-8\\ \hline
10&-12&-14&16\\ \hline
-1&3&9&-11\\ \hline
5&-7&-13&15\\ \hline
\end{array} 
$$
$$
\begin{array}{|c|c|c|c|}
\multicolumn{4}{c} {B(4,4)} \\
\hline
9&-1&2&-6\\ \hline
-13&5&-10&14\\ \hline
15&-11&8&-16\\ \hline
-7&3&-4&12\\ \hline
\end{array} \hspace{.2in}
\begin{array}{|c|c|c|c|}
\multicolumn{4}{c} {B(4,8)} \\
\hline
-6&8&4&-2\\ \hline
10&-12&-16&14\\ \hline
5&-7&13&-15\\ \hline
-1&3&-9&11\\ \hline
\end{array} \hspace{.2in}
\begin{array}{|c|c|c|c|}
\multicolumn{4}{c} {B(4,12)} \\
\hline
-1&9&2&-6\\ \hline
5&-13&-10&14\\ \hline
-3&7&4&-12\\ \hline
11&-15&-8&16\\ \hline
\end{array}  \hspace{.2in}
\begin{array}{|c|c|c|c|}
\multicolumn{4}{c} {B(8,8)} \\
\hline
-1&5&-9&13\\ \hline
10&-14&2&-6\\ \hline
11&-7&3&-15\\ \hline
-12&8&-4&16\\ \hline
\end{array} 
$$
$$
\begin{array}{|c|c|c|c|}
\multicolumn{4}{c} {B(8,12)} \\
\hline
-2&6&8&-4\\ \hline
10&-16&-14&12\\ \hline
-5&1&7&-11\\ \hline
9&-3&-13&15\\ \hline
\end{array} \hspace{.2in}
\begin{array}{|c|c|c|c|}
\multicolumn{4}{c} {B(12,12)} \\
\hline
13&-5&-10&14\\ \hline
-9&1&2&-6\\ \hline
-7&3&4&-12\\ \hline
15&-11&-8&16\\ \hline
\end{array} 
$$  \end{proof}


\begin{theorem} \label{from-boosters} There exists an $H(n;k)$ for every $n \geq 8$ where $n \equiv 0$ (mod 4), 
$k \equiv 1$ (mod 4) and   $ 4  \lceil{(n-4)/12}\rceil +5 \leq k <n$ .\end{theorem}

\renewcommand{\arraycolsep}{2pt}
\begin{proof} Let $m = n/4$. Define $$A = 
\begin{array}{|c|c|c|c|} \hline
-1&4&-11&9\\ \hline
-6&8&10&-14\\ \hline
3&-12&13&-7\\ \hline
5&-2&-15&16\\ \hline
\end{array} \hspace{.2in} \renewcommand{\arraycolsep}{2pt}
{\rm and} \hspace{.2in} 
Z_i = \begin{array}{|c|c|c|c|} \hline
-1-4i&&&\\ \hline
&2+4i&&\\ \hline
&&3+4i&\\ \hline
&&&-4-4i\\ \hline
\end{array} 
$$ 

\noindent
for $i= 0,1,\ldots ,n-1$.
Note that $A$ is shiftable and that in $A$, the sum of the elements in row $t$ (and also in column $t$) for $t = 1,2,3,4$ is  $1,-2,-3,4$, respectively.

We begin with an empty $ m \times m$ array (indexed by $0,1, \ldots , m-1$) called $H$.    For $i= 0,1,\ldots ,m-1$ place the array $Z_i$ in cell $(i,i)$ of $H$ and place the arrays $A \pm (n+16i)$ in the cells  $(i,i+1)$ (arithmetic on the rows and columns is modulo $m$).  
So $H$ has 5 filled cells in each row and each column.  
The support of $H$ is $\{1,2, \ldots ,5n\}$ since the support of the diagonal cells is $\{ 1,2, \ldots ,4m=n\}$  and the support of the other cells is $\{n+1, n+2, \ldots ,n + 16m = 5n\}$.

We see that in the four rows created from row $i$ of $H$ the row sums are now $-4i,4i,4i,-4i$.  Similarly, in the four columns created from column $i$ of $H$ the column sums are  $-4i,4i,4i,-4i$. 
This implies that the maximum sum in any row is $4\times (m-1) = 4m-4 =n-4$.   Let $t = \lceil{(n-4)/12}\rceil$ and note that if $0\leq 4s \leq n-4$, then one can write $4s$ as the sum of exactly $t$ values from the set $\{0,4,8,12\}$.  

We will now place shifts of the $B(a,b)$ arrays from Lemma \ref{boosters} in such a way that each row and column sum is zero.
Consider row $i$ of the array $H$.  We will place shifts of the $B(a,b)$ arrays in the cells $(i,i+2),  (i,i+3), \ldots , (i,i +t + 1)$ (so  there will be  $t$ of the $B(a,b)$ arrays in every row and every column).   Label the $B(a,b)$ arrays in this row as  $B_i(a_1,b_1), B_i(a_2,b_2), \ldots ,B_i(a_t,b_t)$.  Now choose the $a_i$'s  from $\{0,4,8,12\}$ so that the sum of the $a_i$'s is $4i$.  This is easy to do and any choice will work (Example \ref{boosterex} below shows one way this can be done).   Now do this for every row.    Do the same process in each column by choosing the appropriate values for the $b_i$'s.  Note the $a_i$'s and the $b_i$'s are independent and that all the necessary arrays $B(a,b)$ exist by Lemma \ref{boosters}.   

It is easy to see that the number of filled cells in each row and in each column is $1 + 4 + 4t
= 4t+5 = 4  \lceil{(n-4)/12}\rceil +5$.  Let $k' = 4  \lceil{(n-4)/12}\rceil +5$.
Finally, since all of the $B(a,b)$ arrays are shiftable, we just shift each one so that  no two are on the same point set and so that the support of the entire final array is $\{1,2, \ldots , nk'\}$.

All that needs to be checked now is that all of the row and column sums are equal to zero.  Consider the sums of the elements in the four new rows created from row $i$ of $H$.  They are the four row sums of the  array $Z_i+ (A \pm (n+16i)) +  B_i(a_1,b_1)+ B_i(a_2,b_2)+ \ldots + B_i(a_t,b_t)$.  But as noted above, the four row sums of 
$Z_i+ (A \pm (n+16i))$ are $-4i,4i,4i,-4i$.  By construction, since the sum $a_1+a_2+ \ldots +  a_t$  equals $4i$, then the four row sums of  $B_i(a_1,b_1)+ B_i(a_2,b_2)+ \ldots + B_i(a_t,b_t)$ are $4i,-4i,-4i,4i$.  Hence the sum of the elements in every row in the resulting final array is zero. This is similarly true for each column.  Hence we have constructed a Heffter array $H(n;k')$, when $k' = 4  \lceil{(n-4)/12}\rceil +5$.  

Let   $k \equiv 1$ (mod 4) with $k'\leq k < n$, then $k = 4s + k'$ for some $s$.  To construct an $H(n;k)$ in this case, begin with the $H(n;k')$ constructed above. Then  in exactly $s$ cells in each row and $s$ cells in each column of the original $H$ array (which are not filled in the above construction) place appropriate shifts of the $B(0,0)$ from Lemma \ref{boosters} so that each symbol from 1 to $nk$ is covered exactly once.    Since the row sums and column sums of any shift of $B(0,0)$ is 0, the resulting array still has row and column sums all equal to zero and now the number of filled cells per row and column is $k$.  This completes the proof.
\end{proof}

The following is an example of the previous theorem.

\renewcommand{\arraycolsep}{.5pt}
\renewcommand{\arraystretch}{2.5}

\begin{example}\label{boosterex} An H(28;13) constructed using Theorem \ref{from-boosters}.
\end{example}

\noindent
In this example,  $k' = 4  \lceil{(n-4)/12}\rceil +5 = 13$.

\renewcommand{\arraycolsep}{2pt}
\renewcommand{\arraystretch}{1.4}

\begin{center}
$\footnotesize 
\begin{array}{|c|c|c|c|c|c|c|}\hline
Z_0&A\pm28&B(0,0)\pm140&B(0,12)\pm156& &&\\ \hline
&Z_1&A\pm44&B(0,0)\pm172&B(4,4)\pm188& &\\ \hline
&&Z_2&A\pm60&B(0,12)\pm204&B(8,8)\pm220& \\ \hline
 &&&Z_3&A\pm76&B(0,12)\pm236&B(12,12)\pm252\\ \hline
B(4,0)\pm284& &&&Z_4&A\pm92&B(12,12)\pm268\\ \hline
B(12,0)\pm300&B(8,4)\pm316&&&&Z_5&A\pm108\\ \hline
A\pm124&B(12,0)\pm332&B(12,8)\pm348&&&&Z_6\\ \hline
\end{array}$
\end{center}
\renewcommand{\arraystretch}{1}

\bigskip
Note that Theorem \ref{from-boosters} constructs $H(n;k)$  in all of the cases roughly when $k > n/3$ ($k \equiv 1$ (mod 4)) for all $n\equiv 0$ (mod 4).
In the cases when  $n \equiv 0$ (mod 12), $n \equiv 0$ (mod 16) and $n \equiv 4$ (mod 16) we can do better than Theorem \ref{from-boosters}. These constructions are presented next.


A set of $n$ $u\times u$ arrays $A_i$, for $i=1,\dots, n$,  is called a {\em set of n $(u;v)$-filler arrays} if they satisfy the following properties:

\begin{enumerate}
  \item for each $i=1,\dots, n$, each row and each column of $A_i$ contains $v$ filled cells,
\item the support of $\cup_{i =1}^n A_i = \{1,\dots,nuv\}$,
\item for each $i=1,\dots, n$, the sum of the elements of each row of $A_i$ is $i$, and
\item for each $i=1,\dots, n$, the sum of the elements of each column of $A_i$ is $i$.
\end{enumerate}

\begin{lemma} \label{strip-lemma} Assume that there exists a strippable Heffter array $H(n;2k+1)$ and a set of  $n$ 
$(u;v)-$filler arrays $A_i$, $i=1,\dots, n$. 
Then there exists a Heffter array $H(nu;v+2k)$.
\end{lemma}

\begin{proof}

Let ${\cal H}$ represent a strippable $H(n;2k+1)$ Heffter array with $T$  representing the cells of the  primary transversal. Without loss of generality assume that $T= \{(i,i)\mid 1\leq i\leq n\}$. We will be replacing the cells of ${\cal H}$ with $u \times u$ filler-arrays and will call our final resulting array $C$. 
Let $A_i$, $i=1,\dots, n$ represent the set of $n$, 
$(u;v)-$filler arrays.  We begin by placing the filler arrays in the cells of the primary transversal in ${\cal H}$.  In particular, if ${\cal H}(i,i)=t>0$, place $A_t$ in cell $(i,i)$ and if ${\cal H}(i,i)=-t < 0$, place $-A_t$ in cell $(i,i)$.  At this stage the support $C$ is  
$\{1,\dots,nuv\}$.  Let $x = nuv$.

We intend to replace the remaining filled cells of ${\cal H}$ (i.e. the filled cells of ${\cal H} \setminus T$) with $u \times u$ arrays.  Since these remaining cells of ${\cal H}$  have support 
$\{ n+1, n+2, \ldots , n(2k+1)\}$ we will name the  $u \times u$ arrays accordingly.
For $i \in \{1, 2, \ldots ,2nk\}$ define the $u \times u$ diagonal array $D_{n+i}$ as follows. The array is empty except for the main diagonal and $D_{n+i}(j,j) = x+ (j-1)2nk +i$, for $1\leq j \leq u$.   Here is a visualization of $D_{n+i}$:
$$ D_{n+i}=
\begin{array}{|cccccc|} \hline
 x+i &&&&& \\
&x+ 2nk +i&&&& \\
&&x+ 2(2nk) +i&&& \\  
&&&\ddots&& \\
&&&&& \\
&&&&& x+ (u-1)2nk +i\\ \hline
\end{array}
$$

To complete the construction we now replace the remaining filled cells in ${\cal H}$ with the $u \times u$ arrays $D_j$ as follows.   If  ${\cal H}(a,b)=t>0$, then  place $D_t$ in cell $(a,b)$ and if ${\cal H}(a,b)=-t<0$, then  place $-D_t$ in cell $(a,b)$.  We have now constructed a $nu \times nu$ array $C$.  We will show that $C$ is a Heffter array $H(nu;v+2k)$.

It is clear that each row and each column of $C$ contains $v$ symbols from an $A_i$ array and $2k$ symbols from the $2k$ $D_j$ arrays, hence each row and each column contains $ v+2k$ filled cells, as required.   Now we look at row sums.  If any $D_i$ and $-D_j$ are concatenated, it is easy to check that each row sum is $i-j$.  The sum of the nondiagonal cells in row $r$ of $\cal H$ is $-{\cal H}(r,r)$ and there are the same number of positive entries in row $r$ as negative entries.  Thus when these entries are replaced by the corresponding $D_i$ arrays, we have that every one of the resulting $u$ rows has row sum equal to $-{\cal H}(r,r)$.  But the $A$ array placed in ${\cal H}(r,r)$ has all row sums equal to ${\cal H}(r,r)$, hence all of the row sums of $C$ are equal to $0$.  The same reasoning shows that all column sums in $C$ also equal $0$.

The proof will be complete if we can show that the support of $C$ is $\{1,2, \ldots ,(nu)(v+2k)\}$.  The support of $\cup_{i =1}^n A_i = \{1,\dots,nuv\}$. The support of the first rows of $D_{n+1}, D_{n+2}, \ldots ,D_{n+2nk}$ is $\{x+1,x+2, \ldots ,x+2nk\}$,  (where $x = nuv$).   In general, for $1\leq j \leq u$, the support of the $j$th rows of $D_{n+1}, D_{n+2}, \ldots ,D_{n+2nk}$ is 
$$\{x+(j-1)2nk +1,x+(j-1)2nk+ 2, \ldots ,x+(j-1)2nk+2nk\}.$$
\noindent\
It is now easy to see that the support of $C$ is 
$$
(\cup_{j=1}^u \{x+(j-1)2nk +1,x+(j-1)2nk+ 2, \ldots ,x+(j-1)2nk+2nk\} )\cup \{1,\dots,nuv\}.$$
\noindent
The above union is equal to $\{1,2,\ldots ,nu(v+2k)\}$.  Hence we have that $C$ is indeed an $H(nu;v+2k)$ as desired.
\end{proof}

In the next two lemmas we construct some sets of filler arrays for use in Lemma \ref{strip-lemma}.

\begin{lemma}\label{(3,3)-filler)}
For all $n\geq 2$, there exists a set of $n$ $(3;3)-$filler arrays.
\end{lemma}

\begin{proof}Below are $n$ arrays $A_t$, $t=1,\dots,n$, which partition the set $\{1,\dots, 9n\}$  into $3\times 3$ arrays such that in each $A_t$ the rows and columns sum to $t$. Specifically, for $t=1, 2, \ldots,  n-1$  the $3\times 3$ arrays $A_t$ are defined as 

$$A_t=\begin{array}{|c|c|c|}
\hline
	9n-t&-8n+t&-n+t\\
\hline
	-6n+t&n+t&5n-t\\
\hline
	-3n+t&7n-t&-4n+t\\
\hline
\end{array} \mbox{\ \ \ \  \   also let \ \ \ \   }
A_n=
\begin{array}{|c|c|c|}
\hline
9n&-6n&-2n\\
\hline
-5n&-n&7n\\
\hline
-3n&8n&-4n\\
\hline
\end{array} \ \ .
$$
It is straightforward to check that the $n$ arrays $A_1, A_2, \ldots ,A_n$ are a set of $n$ $(3;3)$-filler arrays.
\end{proof}

\begin{lemma}\label{(4,3)-filler)}
For all $n\geq 2$, there exists a set of $n$ $(4;3)-$filler arrays.
\end{lemma}
\begin{proof}
Below are $n$ arrays $A_t$, $t=1,\dots,n$, which partition the set $\{1,\dots, 12n\}$  into $4\times 4$ arrays (with one empty cell per row and per column) such that in each $A_t$ the rows and columns all sum to $t$. For $t=1, 2, \ldots,  n-1$ the $4 \times 4$ arrays  $A_t$ are defined as 

\renewcommand{\arraycolsep}{2pt}
\begin{center}
$A_t=
\begin{array}{|c|c|c|c|}
\hline
	&-11n+t&11n+t&-t\\
\hline
	7n-t&&-2n+t&-5n+t\\
\hline
	-9n+t&4n-t&&5n+t\\
\hline
	2n+t&7n+t&-9n-t&\\
\hline
\end{array}
\mbox{\ \ \ \ \ also  let   \ \  }
A_n=
\begin{array}{|c|c|c|c|}
\hline
	&11n &-12n&2n\\
\hline
	-8n&&4n&5n\\
\hline
	10n&-3n&&-6n\\
\hline
	-1n&-7n&9n&\\ \hline
\end{array} \ \ . $
\end{center}

\noindent
It is again easy to check that the $n$ arrays  $A_1, A_2, \ldots ,A_n$ are a set of $n$ $(4;3)$-filler arrays.
\end{proof}

\renewcommand{\arraycolsep}{4pt}

\begin{corollary}\label{12m:5}
There exists $H(12m;5)$ for all $m\geq 1$.
\end{corollary}

\begin{proof} For all $m \geq 1$ there exists a strippable $H(4m;3)$ by Corollary \ref{strip-4m,3}.  By Lemma  \ref{(3,3)-filler)} there is a set of $4m$ (3;3)-filler arrays.  The result then follows from Lemma \ref{strip-lemma}.
\end{proof}

We demonstrate the smallest case of Corollary \ref{12m:5} by constructing an $H(12;5)$.  

\begin{example} \label{H12,5}
An $H(12;5)$ constructed via Corollary \ref{12m:5}
\end{example}

We begin with the $H(4;3)$ constructed from Theorem \ref{4m+1,3} (which was also presented in Example \ref{ex1}). Note the  primary transversal  on the main diagonal.

\begin{center}
\renewcommand{\tabcolsep}{4pt}
\begin{tabular}{|c|c|c|c|}  \hline
4& 8 & & -12\\ \hline
-9 & 3 & 6 &  \\ \hline
 & -11 & 1 & 10 \\ \hline
5 &  & -7 & 2 \\ \hline 
\end{tabular}
\end{center}

\noindent
From Lemma \ref{(3,3)-filler)}, there are four $(3;3)-$filler arrays $A_1, A_2, A_3$ and $A_4$. They are as follows:
\renewcommand{\tabcolsep}{12pt}
\begin{center}
\begin{tabular}{cccc}
$A_1$&$A_2$&$A_3$ & $A_4$\\
\renewcommand{\tabcolsep}{4pt}
\begin{tabular}{|c|c|c|}
\hline
35&-31&-3\\ \hline
-23&5&19\\ \hline
-11&27&-15\\ \hline
\end{tabular}& \renewcommand{\tabcolsep}{4pt}
\begin{tabular}{|c|c|c|}
\hline
34&-30&-2\\ \hline
-22&6&18\\ \hline
-10&26&-14\\ \hline
\end{tabular}
& \renewcommand{\tabcolsep}{4pt}
\begin{tabular}{|c|c|c|}
\hline
33&-29&-1\\ \hline
-21&7&17\\ \hline
-9&25&-13\\ \hline
\end{tabular}
& \renewcommand{\tabcolsep}{4pt}
\begin{tabular}{|c|c|c|}
\hline
36&-24&-8\\ \hline
-20&-4&28\\ \hline
-12&32&-16\\ \hline
\end{tabular}
\end{tabular}
\end{center}
\renewcommand{\tabcolsep}{4pt}

\noindent
Note that in Lemma \ref{strip-lemma}, $4 \times 3 \times 3 = 36$. To construct the $H(12;5)$  via Lemma \ref{strip-lemma} we first show the step where each symbol of the 
$H(4;3)$ is replaced by either an $A$ array or a $D$ array.  We get the following array.

\begin{center} 
$
\begin{array}{|c|c|c|c|}  \hline
A_4& D_8 & & -D_{12} \\ \hline
-D_{9} & A_3 & D_{6} &  \\ \hline
 & -D_{11} & A_1 & D_{10} \\ \hline
D_{5} &  & -D_{7} & A_2 \\ \hline
\end{array} $
\end{center}

\noindent
The final $H(12;5)$ is below.

\begin{center}
\begin{tabular}{|c|c|c|c|c|c|c|c|c|c|c|c|c|}
\hline
36&-24&-8&40&&&&&&-44&&\\ \hline
-20&-4&28&&48&&&&&&-52&\\ \hline
-12&32&-16&&&56&&&&&&-60\\ \hline
-41&&&33&-29&-1&38&&&&&\\ \hline
&-49&&-21&7&17&&46&&&&\\ \hline
&&-57&-9&25&-13&&&54&&&\\ \hline
&&&-43&&&35&-31&-3&42&&\\ \hline
&&&&-51&&-23&5&19&&50&\\ \hline
&&&&&-59&-11&27&-15&&&58\\ \hline
37&&&&&&-39&&&34&-30&-2\\ \hline
&45&&&&&&-47&&-22&6&18\\ \hline
&&53&&&&&&-55&-10&26&-14 \\ \hline 
\multicolumn{12}{c}{$H(12;5)$}
\end{tabular}\end{center}

\begin{theorem} \label{12m;k} There exists an $H(12m;k)$ for every $m \geq 1$  and every   
$k \equiv 1$ (mod 4) with  $5 \leq k < 12m$ .\end{theorem}
\begin {proof} Use Corollary \ref{12m:5} to make an $H(12m;5)$.  This array was constructed from an 
$H(4m;3)$ which had a primary diagonal on the main diagonal and filled cells on the two diagonals adjacent to the main diagonal.  It is easy to see that all of the filled cells of the resulting $H(12m;5)$ occur in 7 consecutive diagonals centered at the main diagonal (see Example \ref{H12,5} above). 
This leaves $12m-7$ consecutive empty diagonals in the resulting array.  Now use Lemma \ref{add4} to construct the desired $H(12m;k)$. \end{proof}

Next we use the $(4;3)-$filler arrays  of Lemma \ref{(4,3)-filler)} to cover additional cases.

\begin{corollary}\label{16m-5}
There exists an $H(16m;5)$ and an $H(16m+4;5)$ for all $m\geq 1$.
\end{corollary}

\begin{proof} For all $m \geq 1$ there exists a strippable $H(4m;3)$ by Corollary \ref{strip-4m,3} and a strippable
$H(4m+1;3)$ by Corollary \ref{4n+1.strip}.  From Lemma  \ref{(4,3)-filler)} there is a set of $4m$  and a set of $4m+1$ (4;3)-filler arrays.  The result again follows from Lemma \ref{strip-lemma}.
\end{proof}

\begin{theorem}\label{h16m} a) There exists an $H(16m;k)$ for every $m \geq 1$  and every  
$k \equiv 1$ (mod 4) with  $5 \leq k < 16m $. 

(b) There exists an $H(16m+4;k)$ for every $m \geq 1$  and  
$k \equiv 1$ (mod 4) with  $5 \leq k < 16m +4 $ 
.\end{theorem}

\begin{proof} (a) From the construction given in Lemma \ref{strip-lemma}, it is  straightforward to see that  the filled cells from the $A_i$'s   (the $(4;3)-$filler arrays) occupy cells in the seven diagonals centered around the main diagonal.  In addition, the cells from the  $D_i$ arrays again appear in the two diagonals that start in cell $(1,5)$ and in cell $(5,1)$.  Hence there are precisely 9 consecutive diagonals (centered at the main diagonal) which contain filled cells. Thus  the longest set of consecutive empty diagonals is  of size $16m - 9$. We use Lemma  \ref{add4} to  fill four diagonals at a time to this square. So we can add at most $16m-12$ new filled cells in each row and each column.  Since there were 5 filled cells per row and column to start with, we can therefore have at most $16m-7$ filled cells in each row and each column of the resulting Heffter array from this construction.   Luckily, from Theorem \ref{from-boosters} there exists an $H(16m;16m-3)$ completing the proof of part (a).

(b) The proof is identical to case (a) except now the array has order $16m+4$ instead of $16m$. Note that again the very largest case of $k$ exists from Theorem \ref{from-boosters}.
\end{proof}

We summarize the results of this section in the next theorem.

\begin{theorem} \label{summary4.1} There exists a Heffter array $H(n;k)$ with $n \equiv 0$ (mod 4) and $k \equiv 1$ (mod 4) if

\begin{tabular}{l}
a) $ 4  \lceil{(n-4)/12}\rceil +5 \leq k <n$, or \\
b)  $n \equiv 0$ (mod 12) and $5 \leq k < n$, or \\
c) $n \equiv 0$ (mod 16)  and   $5 \leq k <n$, or \\
d) $n \equiv 4$ (mod 16)  and   $5 \leq k < n $. 
\end{tabular}
\end{theorem}

\begin{proof}  Part {\em a} is Theorem \ref{from-boosters}, part $b$ is Theorem \ref{12m;k} and parts $c$ and $d$ are  Theorem \ref{h16m}.
\end{proof}

\subsection{{\em H(n;k)} with {\em n} $\equiv$ 3 (mod 4) and {\em k} $\equiv$ 1 (mod 4)} 
\label{subsection4.2}

We begin with the two smallest example in this case, an $H(7;5)$ and an $H(11;5)$.

\renewcommand{\tabcolsep}{2pt}
\begin{example} An $H(7;5)$ and an $H(11;5)$.
\end{example}
{\small
\begin{center}
\begin{tabular}{|r|r|r|r|r|r|r|} \hline
$-$10&&16&$-$1&$-$2&$-$3& \\ \hline
&$-$4&&$-$6&$-$7&$-$5&22 \\ \hline
$-$30&29&$-$9&$-$8&&&18\\ \hline
$-$11&&$-$12&28&$-$31&26& \\ \hline
&$-$14&$-$15&$-$13&&17&25 \\ \hline
27&$-$34&20&&19&&$-$32 \\ \hline
24&23&&&21&$-$35&$-$33 \\ \hline
\multicolumn{7}{c}{$H(7;5)$}\\
\end{tabular}\hspace{.4in}
\begin{tabular}{|r|r|r|r|r|r|r|r|r|r|r|} \hline
$-$1&$-$2&$-$3& & &37& & & &$-$31& \\ \hline
$-$4&$-$5& &$-$6& & &$-$23&38& & & \\ \hline
$-$7&$-$8& &$-$18&$-$10& & & &43& & \\ \hline
 & &$-$11&$-$16&$-$9&$-$17&53& & & & \\ \hline
 & &$-$14&$-$12&$-$13&$-$15& & & & &54\\ \hline
 &40& & &$-$19&$-$49&$-$20& & &48& \\ \hline
 & &$-$22&52& & &$-$55&$-$21&46& & \\ \hline
39&$-$25& & & & & &$-$24& &42&$-$32\\ \hline
$-$27& & & & & &45&41&$-$26& &$-$33\\ \hline
 & & & & &44& &$-$34&$-$28&$-$29&47\\ \hline
 & &50& &51& & & &$-$35&$-$30&$-$36\\ \hline
\multicolumn{11}{c}{$H(11;5)$}
\end{tabular}
\end{center}
} 

In a manner similar to what was done in Corollary \ref{12m:5} we will use strippable Heffter arrays $H(4m+1;3)$ and (3;3)-filler arrays to obtain $H(12m+3;5)$ for every $m \geq 1$.  We will then construct $H(12m+3;k)$ with $k \equiv 1$ (mod 4) for all $5\leq k <12m+3$ via Lemma \ref{add4}.   We do this in the following two theorems.

\begin{theorem} \label{12m+3;5}
There exists a $H(12m+3;5)$ for all $m\geq 1$.
\end{theorem}

\begin{proof}From Corollary \ref{4n+1.strip} there exists a strippable Heffter array $H(4m+1;3)$ for all $m \geq 1$.
 By Lemma  \ref{(3,3)-filler)} there is a set of $4m+1$ (3;3)-filler arrays.  The result now follows from Lemma \ref{strip-lemma}.
\end{proof}

\begin{theorem} There exists an $H(12m+3;k)$ for every $m \geq 1$ and every  
$k \equiv 1$ (mod 4) with  $5 \leq k < 12m+3$ .\end{theorem}

\begin {proof} Use Theorem \ref{12m+3;5} to make $H(12m+3;5)$ for all $m\geq 1$.    
As in Theorem \ref{12m;k} we see that the filled cells all occur in 7 consecutive diagonals.  This leaves $12m+3-7=12m-4$ consecutive empty diagonals in the resulting array.  Now use Lemma \ref{add4} to construct the desired $H(12m+3;k)$. \end{proof}

Unfortunately we can not present any general results in the two remaining subcases of $H(n;k)$ with $k\equiv 1$ (mod 4), namely when $n \equiv 7$ (mod 12) (except $n=7$) or when  $n \equiv 11$ (mod 12).

\section {Conclusion} \label{section5}

We have constructed square integer Heffter arrays $H(n;k)$ for many of the possible orders. 
Below is a table showing all of the cases that have been considered in this paper.  Note that by Theorem
\ref{necessary} an integer $H(n;k)$ does not exist unless $nk \equiv 0,3$ (mod 4). For the cases that can exist we give the theorem number that proves the existence. In the cases we have not solved completely, a subsection is given that contains the partial result.  Note $n$ and $k$ represent congruence classes modulo 4.  We do not hesitate to conjecture that there exists an integer $H(n;k)$ if and only if  $n \geq k\geq 3$ and $nk \equiv 0,3$ (mod 4).

\renewcommand{\tabcolsep}{4pt}

\begin{center}
\begin{tabular} {|c||c|c|c|c|} \hline 
$n\backslash k$&0&1&2&3\\ \hline\hline
0&Corollary \ref{4k}&   

\begin{tabular}{c}
 $ 4  \lceil{(n-4)/12}\rceil +5 \leq k <n$\\
 $n \equiv 0$ (mod 12) and $5 \leq k < n$ , or \\
$n \equiv 0$ (mod 16)  and   $5 \leq k < n  $\\
$n \equiv 4$ (mod 16)  and   $5 \leq k < n  $\\
see Subsection \ref{subsection4.1}
\end{tabular}

&Theorem \ref{th1}&Theorem \ref{ladder.4m} \\ \hline
1&Corollary \ref{4k}&DNE&DNE&Theorem \ref{ladder.4m+1} \\ \hline
2&Corollary \ref{4k}&DNE&Theorem \ref{th1}&DNE \\ \hline
3&Corollary \ref{4k}  &

\begin{tabular}{c}

$n \equiv 3$ (mod 12) and $5 \leq k < n$ \\
see Subsection \ref{subsection4.2}
\end{tabular}

&DNE&DNE \\ \hline
\end{tabular}
\end{center}

\section{Acknowledgement} Much of this research was done at the University of Queensland while the second  and fourth authors were visiting  (at separate times).  We thank  the Ethel Raybould Visiting Fellowship and the Turkish Science Foundation 2219 Program for providing funds for these visits.


\end{document}